\documentclass[11pt]{amsart}

\usepackage{amsfonts,amsmath}
\usepackage{amssymb, latexsym}
\usepackage{amscd,amsthm}
\input xy
\xyoption{all}

\newcommand{\marginlabel}[1]%
  {\mbox{}\marginpar{\raggedleft\hspace{0pt}\bfseries\sf#1}}

\def\ZZ{{\mathbf Z}}

\def\CC{{\mathbf C}}
\def\PP{{ \mathbf{P}}}

\def\OO{{\mathcal O}}

\def\L{\mathbf{L}}

\def\F{\mathcal{F}}
\def\E{\mathcal{E}}
\def\P{\mathcal{P}}

\def\Pic0{{\rm Pic}^0(X)}

\newcommand{\alb}{\textnormal{alb}}
\newcommand{\tn}[1]{\textnormal{#1}}
\newcommand{\Alb}{\textnormal{Alb}}
\newcommand{\lra}{\longrightarrow}
\newcommand{\noi}{\noindent}
\newcommand{\PPP}{\PP_{\text{sub}}}
\newcommand{\RR}{\mathbf{R}}

\newcommand{\FF}{\mathcal{F}}

\newcommand{\rk} {\text{rank }}

\newcommand{\HH}[3]{H^{{#1}}  ( {#2} , {#3}
 ) }
\newcommand{\hh}[3]{h^{{#1}} \ ( {#2} , {#3}
 ) }
\newcommand{\Hog}[2]{\Gamma   ( {#1} , {#2}
 ) }
\newcommand{\HHnospace}[2]{H^{{#1}}  (
{#2}  )  }

\newcommand{\Picc}{\textnormal{Pic}}
\newcommand{\AAA}{\mathbf{A}}
\newcommand{\VVV}{\mathbf{V}}
\newcommand{\linser}[1]{| \, {#1} \, |}
\newcommand{\Shom}{\mathcal{H}om}
\newcommand{\Shomm}[1]{\Shom\big( {#1} \big)}
\newcommand{\codim}{\tn{codim}\, }
\newcommand{\pr}{\prime}

\theoremstyle{plain}

\newtheorem{theorem}{Theorem}[section]
\newtheorem{theoremalpha}{Theorem}

\newtheorem{proposition/example}[theorem]{Proposition/Example}
\newtheorem{proposition}[theorem]{Proposition}
\newtheorem{corollary}[theorem]{Corollary}
\newtheorem{lemma}[theorem]{Lemma}

\newtheorem{variant}[theorem]{Variant}
\newtheorem{conjecture}[theorem]{Conjecture}

\theoremstyle{definition}
\newtheorem{definition}[theorem]{Definition}
\newtheorem{remark}[theorem]{Remark}
\newtheorem{example}[theorem]{Example}

\newtheorem{conjecture/question}[theorem]{Conjecture/Question}

\newtheorem{remark/definition}[theorem]{Remark/Definition}
\newtheorem{definition/notation}[theorem]{Definition/Notation}

 \theoremstyle{remark}

\numberwithin{equation}{section}

\begin{document}

\title[Derivative Complex, BGG Correspondence, and Numerical Inequalities]{Derivative Complex, BGG Correspondence,  and  Numerical Inequalities for Compact K\"ahler Manifolds}

\author {Robert Lazarsfeld}
\address{Department of Mathematics, University of Michigan,
530 Church Street, Ann Arbor, MI 48109, USA } \email{{\tt
rlaz@umich.edu}}
\thanks{First author partially supported by NSF grant DMS-0652845}

\author{Mihnea Popa}
\address{Department of Mathematics, University of Illinois at Chicago,
851 S. Morgan Street, Chicago, IL 60607, USA } \email{{\tt
mpopa@math.uic.edu}}
\thanks{Second author partially supported by NSF grant DMS-0758253 and a Sloan Fellowship}

\maketitle
\setlength{\parskip}{0 in}


\setlength{\parskip}{.07 in}


\section*{Introduction}

Given an irregular compact K\"ahler manifold $X$, one can form the \textit{derivative complex} of $X$, which governs the deformation theory of the groups $\HH{i}{X}{\alpha}$ as $\alpha$ varies over $\Picc^0(X)$. Together with its variants, it plays a central role in a body of work involving generic vanishing theorems. The purpose of this paper is to present two new applications of this complex. First, we show that it fits neatly into the setting of the so-called  Bernstein-Gel'fand-Gel'fand (BGG) correspondence between modules over an exterior algebra and linear complexes over a symmetric algebra (cf.\ \cite{bgg}, \cite{efs}, \cite{eisenbud}, \cite{coanda}). What comes out here is that some natural cohomology modules associated to $X$ have a surprisingly simple algebraic structure, and that conversely one can read off from these modules some basic geometric invariants associated to the Albanese mapping of $X$. Secondly, under an additional hypothesis, the derivative complex determines a vector bundle on the projectivized tangent space to $\Picc^0(X)$ at the origin. We show that this bundle, which encodes the infinitesimal behavior of the Hilbert scheme of paracanonical divisors along the canonical linear series, can be used to generate  inequalities among numerical invariants of $X$.

Turning to details, we start by introducing the main homological players in our story. Let $X$ be a compact K\"ahler manifold of dimension $d$, with $\HH{1}{X}{\OO_X} \ne 0$,  
and let $\PP= \PPP\big( \HH{1}{X}{\OO_X}\big)$ be the projective space of one-dimensional subspaces of $\HH{1}{X}{\OO_X}$. Thus a point in $\PP$ is given by a non-zero vector $v \in \HH{1}{X}{\OO_X}$, defined up to scalars. Cup product with $v$ determines a complex $\underline{\mathbf{ L}}_X$ of vector bundles on $\PP$:  
\begin{multline} \label{BGG.Sheaf.Cx.Intro}
 0\lra \OO_{\PP}(-d)\otimes \HH{0}{X}{\OO_X} \lra  \OO_{\PP} (-d+1)\otimes \HH{1}{X}{\OO_X}\lra  \ldots \\ \ldots \lra \OO_{\PP} (-1) \otimes \HH{d-1}{X}{\OO_X} \lra \OO_{\PP} \otimes \HH{d}{X}{\OO_X} \rightarrow 0.
  \end{multline}
Letting $S = \tn{Sym}\big( \HH{1}{X}{\OO_X}^\vee \big)$ be the symmetric algebra on the vector space $\HH{1}{X}{\OO_X}^\vee$, taking global sections in 
$\underline{\mathbf{ L}}_X$  gives rise to 
 a linear complex ${\bf L}_X$ of graded $S$-modules  in homological degrees $0$ to $d$: 
\begin{equation} \label{BGG.Complex.Intro}
0\longrightarrow S\otimes_{\CC} \HH{0}{X}{\OO_X} \longrightarrow  S\otimes_{\CC} \HH{1}{X}{\OO_X}\longrightarrow \ldots \longrightarrow S\otimes_{\CC} \HH{d}{X}{\OO_X} \lra 0. 
\end{equation}
These two complexes are avatars of the \textit{derivative complex} introduced in \cite{gl2} and studied for instance in \cite{hacon}, \cite{pareschi}, \cite{ch}, \cite{PP2}, which computes locally the pushforward to $\Picc^0(X)$ of the Poincar\'e line bundle on $X \times \Picc^0(X)$.
    We shall  also be interested in the coherent sheaf $\FF = \FF_X$ on $\PP$ arising as  the cokernel of the right-most map in the complex $\underline{\mathbf{ L}}_X$, so that one has an exact sequence:
\begin{equation} \label{BGG.Sheaf.Eqn.Intro}
\OO_{\PP} (-1) \otimes \HH{d-1}{X}{\OO_X} \lra \OO_{\PP} \otimes \HH{d}{X}{\OO_X} \rightarrow \FF\lra 0.
\end{equation}
For reasons that will become apparent shortly, we call ${\underline{\bf L}}_X$ and ${\bf L}_X$ the BGG-\textit{complexes} of $X$, and $\FF_X$ its BGG-\textit{sheaf}. 

We shall be concerned with the exactness properties of ${\underline{\bf L}}_X$ and ${\bf L}_X$.
Let
\[  \alb_X : X \lra \Alb(X)\] be the Albanese mapping of $X$,
and let 
\[ k \ = \ k(X) \ = \ \dim X \,  - \, \dim \alb_X(X) \]
be the dimension of the general fiber of $\alb_X$. 
We say that $X$ carries an \textit{irregular fibration} if it admits a surjective morphism $  X \lra Y$ with connected positive dimensional fibres onto a normal analytic variety $Y$ with the property that (any smooth model of) $Y$ has maximal Albanese dimension. These are the higher-dimensional analogues of irrational pencils in the case of surfaces.  The behavior of ${\bf L}_X$ and $\underline{{\bf L}}_X$ is  summarized in the following technical statement, which pulls together results from the papers cited above.

\begin{theoremalpha}\label{GV_cons}\begin{enumerate}
\item[(i).] The complexes $\mathbf{L}_X$ and  $\underline{\mathbf{ L}}_X$ are  exact in the first $d-k$ terms from the left, but $\mathbf{L}_X$ has non-trivial homology at the next term to the right.
\vskip 5pt
 \item[(ii).] 
 Assume that $X$ does not carry any irregular fibrations. Then the \textnormal{BGG }sheaf $\FF$ is a vector bundle on $\PP$ with
 $ \tn{rk}(\FF)= \chi (\omega_X)$, and  $\underline{\mathbf{ L}}_X$ is a resolution of $\FF$.
\end{enumerate}
\end{theoremalpha}
\noi The essential goal of the present paper is to show that these exactness properties have some interesting consequences for the algebra and geometry of $X$.

The first applications concern
the  cohomology modules
\[ P_X \, =  \, \underset{i = 0}{\overset{d}{\oplus}} ~ \HH{i}{X}{ \OO_X} \ \ , \ \ 
Q_X \, =  \, \underset{i = 0}{\overset{d}{\oplus}}~ H^i (X, \omega_X) 
\] 
of the structure sheaf and the canonical bundle of $X$.
Via cup product, we may view these as graded modules over the exterior algebra 
\[ E \ =_{\text{def}} \  \Lambda^* \HH{1}{X}{\OO_X}\] on $\HH{1}{X}{\OO_X}$.\footnote{Following the degree conventions of \cite{efs} and \cite{eisenbud}, we take $E$ to be generated in degree $-1$, and then declare that the summand   $\HH{i}{X}{\omega_X}$ of   $Q_X$ has degree  $-i$, while $\HH{i}{X}{\OO_X}$ has degree $d-i$ in 
$P_X$.}
There has been a considerable amount of recent work in the commutative algebra community aimed at extending to modules over an exterior algebra aspects of the classical theory of graded modules over a polynomial ring.  In  the present context, it is natural to ask whether one can say anything in general about the algebraic properties of the modules $P_X$ and $Q_X$ canonically associated to 
a K\"ahler manifold $X$: for instance, in what degrees do generators and relations live?
  
  An elementary example might be helpful here. Consider an abelian variety $A$ of dimension $d +1$, and let $X \subseteq A$ be a smooth hypersurface of very large degree. Then by the Lefschetz theorem one has
  \begin{align*}
\HH{i}{X}{\OO_X} \ &= \ \HH{i}{A}{\OO_A}\  = \ \Lambda^i  \HH{1}{X}{\OO_X} \ \text{ for } i < d \\
\HH{d}{X}{\OO_X} \ &\supsetneqq \ \HH{d}{A}{\OO_A} \  = \ \Lambda^d \HH{1}{X}{\OO_X}.
  \end{align*}
Thus $P_X$ has generators as an $E$-module in two degrees:  $1 \in \HH{0}{X}{\OO_X}$, and many new generators in $\HH{d}{X}{\OO_X}$. If one takes the viewpoint that the  simplest  $E$-modules are those whose generators appear in a single degree, this means that $P_X$ is rather complicated. On the other hand, the situation with the dual module $Q_X$ is quite different. Here $\HH{0}{X}{\omega_X}$ is big, and the  maps
\[  \HH{0}{X}{\omega_X} \otimes \Lambda^i \HH{1}{X}{\OO_X} \lra \HH{i}{X}{\omega_X} \]
are surjective, i.e. $Q_X$ is generated in degree $0$. This suggests that  $Q_X$ behaves more predictably as an $E$-module than does $P_X$. 
Our first main result asserts that quite generally the module $Q_X$ has simple homological properties. 

 Specifically, given a graded $E$-module $M = \oplus_{i=0}^{-d} M_i$ generated in degrees $\le 0$, one says that $E$ is $m$-regular if the generators of $M$ appear in degrees $0, -1, \ldots, -m$, the relations among these generators are in degrees $-1, \ldots, -(m+1)$, and more generally the $p^{\text{th}}$ module of syzygies of $M$ has all its generators in degrees $\ge -(p+m)$. This is the analogue for modules over the exterior algebra of the familiar notion of Castelnuovo-Mumford regularity for graded modules over a polynomial ring. As in  the classical case, one should see regularity as being a measure of algebraic complexity, with small regularity corresponding to low complexity.  For example, in the example above (an appropriate shift of) the module $P_X$ has worst-possible regularity $= d$. 
 
 The following result asserts that the regularity of $Q_X$ is computed by the Albanese fiber-dimension of $X$.
 \begin{theoremalpha} \label{Intro.Reg.Thm}
As above, let $X$ be a compact K\"ahler manifold, and let $ k =\dim X - \dim \alb_X(X)$. Then
\[   \textnormal{reg}(Q_X) \ = \ k, \]
 i.e. $Q_X$ is $k$-regular, but not $(k-1)$-regular as an $E$-module.  In particular, $X$ has maximal Albanese dimension \tn{(}i.e.\ $k = 0$\tn{)} if and only if  $Q_X$ is generated in degree $0$ and has a linear free resolution.
\end{theoremalpha}
\noi  
According to the general BGG-correspondence, which we quickly review in \S 2, the regularity of a graded module over the exterior algebra $E$ is governed by the exactness properties of a linear complex of modules over a symmetric algebra. Our basic observation is that for the module $Q_X$, this complex is precisely  ${\bf L} _X$. Then Theorem \ref{Intro.Reg.Thm} becomes an immediate consequence of statement (i) of Theorem \ref {GV_cons}.   Note that it follows from  Theorem \ref{Intro.Reg.Thm} that the Albanese dimension of $X$ is determined by purely algebraic data encoded in $Q_X$; it would be interesting to know whether this has any applications.

Our second line of application for Theorem \ref{GV_cons} is as a mechanism for generating inequalities on numerical invariants of $X$.  The search for relations among the Hodge numbers of an irregular variety has a long history, going back at least as far as the classical theorem of Castelnuovo and de Franchis giving a lower bound on the holomorphic Euler characteristic of surfaces without irrational pencils.     Assuming that $X$ does not have any irregular fibrations,  
statement (ii) of Theorem \ref{GV_cons}  implies that the BGG-sheaf $\FF$ is a vector bundle on the projective space $\PP$,  whose invariants are determined by $\underline{\mathbf{L}}_X$. Geometric facts about vector bundles on projective spaces then give rise to inequalities for $X$. Specifically, consider the formal power series:
 \begin{equation} \label{power.series}
\gamma  (X; t) \ =_{\text{def}} \  \prod_{j = 1}^{d} (1 - jt)^{(-1)^j  h^{d, j}} \ \in \ \ZZ[[t]], 
\end{equation} 
where $h^{i,j} = h^{i,j}(X)$. Write $q = \hh{1}{X}{\OO_X}$ for the irregularity of $X$ (so that $q = h^{d,d-1}$), and for  $1 \le i \le q-1$ denote by 
\[ \gamma_i  = \gamma_i(X) \ \in\  \ZZ\]  the coefficient of $t^i$ in $\gamma(X;t)$. Thus $\gamma_i$ is a polynomial in the $h^{d,j}$.  We prove:
\begin{theoremalpha} \label{Intro.Chern.Ineq.Thm}
Assume that $X$ does not carry any irregular fibrations \tn{(}so that in particular $X$ itself has maximal Albanese dimension\tn{)}. Then  \begin{enumerate} \item[(i).] Any Schur polynomial of weight $\le q-1$ in the $\gamma_i$ is non-negative. In particular \[ \gamma_i(X) \ \ge \ 0\] for every $1 \le i \le q-1$. 
\vskip 4pt
\item[(ii).] If $i$ is any index with $\chi (\omega_X) < i < q$, then $\gamma_i(X) = 0$. 
\vskip 4pt
\item[(iii).] One has $\chi (\omega_X) \ge q  - d$. \end{enumerate}
\end{theoremalpha}
\noi
Part (i) expresses in particular the fact that the Chern classes of $\mathcal{F}$ are non-negative.  
For example, when $i = 1$ this yields (under the assumption of the theorem) the inequality
\begin{equation}
 h^{d,1} - 2 h^{d,2} +  3 h^{d,3} - \ldots + (-1)^{d +1}\cdot d \cdot  h^{d,d} \ \ge \ 0. \tag{*}
\end{equation}
 This includes some classically known statements (for instance if $\dim X = 3$, (*) reduces to the Castelnuovo--de Franchis-type inequality $h^{0,2} = h^{3,1} \ge 2q - 3$), but the positivity of higher $\gamma_i$ and part (iii)
produce new stronger results. In fact, for threefolds satisfying the hypotheses of the theorem, an inequality  kindly provided by a referee and the inequality in (iii) together imply that asymptotically
$$ h^{0,2} \ \succeq \ 4q \ {\rm ~and~} \ h^{0,3} \ \succeq \ 4q$$
while in the case of fourfolds the same plus the inequality $\gamma_2 \ge 0$ give asymptotically
$$ h^{0,2} \ \succeq \ 4q  \ , \
h^{0,3} \ \succeq \ 5q + \sqrt{2q} \ {\rm ~and~} \
h^{0,4} \ \succeq \ 3q + \sqrt{2q}.$$
(See Corollary \ref{hodgethreefold} for more details.) 
When $X$ is a surface without irrational pencils, related methods applied to $\Omega_X^1$ yield a new inequality for $h^{1,1}$ as well. 
Assertion (iii) is (a slightly special case of) the main result of \cite{PP2}, for which we provide a simple proof. The idea is that when the rank $\tn{rk}(\mathcal{F}) = \chi(\omega_X)$ of $\mathcal{F}$ is small compared compared to $q -1 = \dim \PP$, it is hard for such a bundle to exist, giving rise to lower bounds on $\chi(\omega_X)$.  
The method used here allows us to further analyze possible borderline cases when the Euler characteristic is small, and conjecture stronger inequalities. All of the inequalities above can fail when $X$ does carry irregular fibrations.

Given the role it plays in \S 3, it is natural to ask what is the geometric meaning of the BGG sheaf $\FF_X$. Our last result,
which could be considered as an appendix to \cite{gl1} \S4 and \cite{gl2}, shows that in fact it has a very natural interpretation. Recall that in classical terminology, a \textit{paracanonical divisor} on $X$ is an effective divisor algebraically equivalent to a canonical divisor.  The set of all such is parametrized by the Hilbert scheme (or Douady space) $\tn{Div}^{\{ \omega \}}(X)$, which admits an Abel-Jacobi mapping
\[  u : \tn{Div}^{\{ \omega \}}(X) \lra \Picc^{\{ \omega \}}(X) 
\]
to the corresponding component of the Picard torus of $X$. The projective space $| \, \omega_X \, |$ parametrizing all canonical divisors sits as a subvariety of $ \tn{Div}^{\{ \omega \}}(X)$: it is the fibre of $u$ over the point $[ \, \omega_X\, ] \in \Picc^{\{ \omega \}}(X) $. 
On the other hand, the projectivization  $\PP(\FF) = \tn{Proj}_{\PP } \big( \tn{Sym}(\FF) \big) $ sits naturally in $\PP^{q-1} \times \PP \big( \HH{d}{X}{\OO_X} \big)$, giving rise to a morphism
\begin{equation} \label{PF.over.Canon.eqn}
\PP(\FF) \lra \PP \big( \HH{d}{X}{\OO_X} \big) \, = \, \linser{\omega_X}.
\end{equation}

\begin{theoremalpha}\label{Intro.Normal.Cone.Thm}
With the notation just introduced, $\PP(\FF)$ is identified via the morphism  \eqref{PF.over.Canon.eqn}
with the projectivized normal cone to $\linser{\omega_X}$ inside $\tn{Div}^{\{ \omega \}}(X)$. 
\end{theoremalpha}
\noi Note that we do not assume here that $X$ carries no irregular fibrations. 
When this additional hypothesis does hold, Theorem \ref{Intro.Normal.Cone.Thm} implies  the amusing fact that whether or not the projective space $|\omega_X|$ is an irreducible component of $ \tn{Div}^{\{ \omega \}}(X)$ depends in most cases only on the Hodge numbers $h^{d,j}(X)$ (Proposition \ref{Exorbitant.Numerical}).

We conclude this Introduction with a few remarks about related work. The sheafified complex $\underline{\bf{L}}_X$ came up in passing in \cite{el} and \cite{hp}, but it has not up to now been exploited in a systematic fashion. Going back to Theorem \ref{Intro.Reg.Thm}, we note that a statement of a similar type was established in \cite{epy} for the singular cohomology of the complement of an affine complex hyperplane arrangement: in fact this cohomology always has a linear resolution over the 
exterior algebra on its first cohomology (though it is not generated in degree zero). However whereas the result of \cite{epy} is of combinatorial genesis and does not involve using the BGG correspondence, as explained above Theorem \ref{Intro.Reg.Thm} is ultimately based on translating Hodge-theoretic information via BGG. We note also that Catanese suggested in \cite{Catan} that it might be interesting to study the BGG correspondence for holomorphic cohomology algebras. In another direction, Lombardi   \cite{lombardi} has extended the approach of the present paper to deal with Hodge groups $\HH{q}{X}{\Omega_X^p}$ with $p,q  \ne 0,d$: see Remark \ref{lombardi} for a brief description of some of his results. 

Concerning the organization of the paper, in \S 1 we prove the main technical result Theorem  \ref{GV_cons}. The connection with BGG is discussed in \S 2. Applications to numerical inequalities occupy \S3, where we also give a number of examples and variants. Finally, Theorem \ref{Intro.Normal.Cone.Thm}
 appears in \S 4.

\medskip

\noindent
{\bf Acknowledgements.} 
The paper owes a lot to discussions with G. Pareschi -- some of the material we discuss would not have existed without his input.
We thank F. Catanese,  I. Coand\u a, Ch. Hacon, M. Mendes Lopes, Ch. Schnell, F.-O. Schreyer and C. Voisin for various comments and suggestions. We are also grateful to one of the referees of an earlier version of this paper for showing us the statement and proof of inequality \eqref{referee} and an alternative proof of Proposition \ref{BGG.Non.Exactness.Prop}, and for allowing us to reproduce them here.


\section{Proof of Theorem \ref{GV_cons}}

This section is devoted to the proof of the main technical result, giving the exactness of the BGG complexes. The argument proceeds in the form of three propositions. We keep notation as in the Introduction: $X$ is a compact K\"ahler manifold of dimension $d$,  
$\alb_X : X \lra \Alb(X)$ is the Albanese mapping, and
\[ k \ = \ k(X) \ = \ \dim X \,  - \, \dim \alb_X(X)  \]
is the dimension of the generic fibre of $\alb_X$. 
As before, $\PP = \PPP\big( \HH{1}{X}{\OO_X}\big)$ is the projective space of lines in $\HH{1}{X}{\OO_X}$, and $\mathbf{L}_X$ and $\underline{\mathbf{L}}_X$ are the complexes appearing in \eqref {BGG.Sheaf.Cx.Intro} and \eqref{BGG.Complex.Intro}.

\begin{proposition} \label{BGG.Exactness.Prop}
The complexes $\mathbf{L}_X$ and $\underline{\mathbf{L}}_X$ are exact in the first $d-k$ terms from the left.
\end{proposition}

\begin{proof}
It is sufficient to prove the exactness for $\mathbf{L}_X$, as this implies the corresponding statement for its sheafified sibling.  The plan for this is to relate  $\L_X$ to the \emph{derivative complex} introduced and studied in \cite{gl2}.

Write $V = \HH{1}{X}{\OO_X}$ and $W = V^*$, and let $\AAA = \tn{Spec(Sym}(W\tn{))}$ be the affine space corresponding to $V$, viewed as an algebraic variety. Thus a point in $\AAA$ is the same as a vector in $V$. Then there is a natural complex $\mathcal{K}^{\bullet}$ of trivial algebraic vector bundles on $\AAA$:
\[
0 \lra \OO_\AAA \otimes \HH{0}{X}{\OO_X} \lra \OO_\AAA \otimes \HH{1}{X}{\OO_X} \lra \ldots \lra \OO_\AAA \otimes \HH{d}{X}{\OO_X} \lra 0,
\]
with maps given at each point of $\AAA$  by wedging with the corresponding element of $V = \HH{1}{X}{\OO_X}$. Recalling that $\Hog{\AAA}{\OO_{\AAA}} = S$, one sees that  $\mathbf{L}_X  = \Hog{\AAA}{\mathcal{K}^{\bullet}}$ is
the complex obtained by taking global sections in $\mathcal{K}^{\bullet}$. As $\AAA$ is affine, to prove the stated exactness properties of $\mathbf{L}_X$, it is equivalent to establish the analogous exactness for the complex $\mathcal{K}^{\bullet}$, i.e. we need to show the vanishings 
$ \mathcal{H}^i({\mathcal{K}^{\bullet}}) = 0 $ of the cohomology sheaves of this complex in the range $i < d-k$. 
For this it is in turn equivalent to prove the vanishing 
\begin{equation}
\mathcal{H}^i\big({\mathcal{K}^{\bullet}}\big)_0 \ = \ 0 \tag{*}
 \end{equation}
 of the stalks at the origin of these homology sheaves in the same range $i < d-k$.  Indeed, (*) implies that $\mathcal{H}^i\big({\mathcal{K}^{\bullet}}\big) = 0$ in a neighborhood of the origin. But the differential of $\mathcal{K}^{\bullet}$ scales linearly in radial directions through the 
origin, so we deduce the corresponding vanishing on all of $\AAA$. 

Now let $\VVV$ be the vector space $V$, considered as a complex manifold, so that $\VVV = \CC^q$, where $q = \hh{1}{X}{\OO_X}$ is the irregularity of $X$.  Then on $\VVV$ we can form as above a complex $(\mathcal{K}^{\bullet})^{an}$\[
0 \lra \OO_\VVV \otimes \HH{0}{X}{\OO_X} \lra \OO_\VVV \otimes \HH{1}{X}{\OO_X} \lra \ldots \lra \OO_\VVV \otimes \HH{d}{X}{\OO_X} \lra 0,
\]
of coherent analytic sheaves, which is just the complex of analytic sheaves determined by the algebraic complex $\mathcal{K}^{\bullet}$. This analytic complex was studied in \cite{gl2}, where it was called the \textit{derivative complex} $D_{\OO_X}^{\bullet}$ of $\OO_X$. Since passing from a coherent  algebraic to a coherent analytic sheaf is an exact functor (cf.\ \cite[3.10]{GAGA}), one has 
$\mathcal{H}^i(({\mathcal{K}^{\bullet}})^{an} )= \mathcal{H}^i(({\mathcal{K}^{\bullet}}) )^{an} $. So it is equivalent for (*) to prove:  
\begin{equation}
\mathcal{H}^i\big(({\mathcal{K}^{\bullet}})^{an}\big)_0 \ = \ 0  \ \ \text{for} \ i < d-k. \tag{**}
 \end{equation}
But this will follow immediately from a body of results surrounding generic vanishing theorems.
 
 Specifically, write $\Picc^0(X) = V/ \Lambda$, let $\P$ be a normalized Poincar\'e line bundle on $X \times \Pic0$,  
and write 
\[ p_1 : X \times \Picc^0(X) \lra X \ \ , \ \ p_2 : X \times \Picc^0(X) \lra \Picc^0(X)\] for the two projections. The main result of \cite{gl2}, Theorem 3.2, says that via the exponential map ${\rm exp}: V \rightarrow \Pic0$ we have the identification of the analytic 
stalks at the origin
\begin{equation}\label{gl_main}
\mathcal{H}^i\big(({\mathcal{K}^{\bullet}})^{an}\big)_0  \  \cong  \ (R^i {p_2}_* \P)_0. \tag{***}
\end{equation}
On the other hand, by \cite{PP2} Theorem C, we have $R^i {p_2}_* \P = 0$ for $i < d-k$, where $p_2$ is the projection onto the second factor.\footnote{This was posed as a problem in \cite{gl2}, and first answered in the smooth projective case by Hacon \cite{hacon} and Pareschi \cite{pareschi}. In \cite{PP2} it was simply shown that the result is equivalent to the Generic Vanishing Theorem of \cite{gl1}, hence it holds also in the compact K\"ahler case.} In view of (***), this gives (**), and we are done.\end{proof}

The next point is the non-exactness of $\mathbf{L}_X$ beyond the range specified in the previous proposition.

\begin{proposition} \label{BGG.Non.Exactness.Prop}
The complex $\mathbf{L}_X$ is not exact at the term $S \otimes_{\CC} \HH{d-k}{X}{\OO_X}$.
\end{proposition}

\begin{remark}
The analogous statement for the sheafified complex $\underline{\mathbf{L}}_X$ can fail. For example, if $X = A \times \PP^k$, with $A$ an abelian variety of dimension $d-k$, then $\underline{\mathbf{L}}_X$ -- which in this case is just a Koszul complex on $\PP^{d-k-1}$ -- is everywhere exact. 
\end{remark}

\begin{proof} [Proof of Proposition \ref{BGG.Non.Exactness.Prop}]\footnote{This proof was indicated to us by one of the referees; it simplifies considerably our original argument.
One can also deduce the statement directly from a theorem of Koll\'ar by passing through the BGG correspondence: see Remark \ref{Streamlined.Non-Exactness.Rmk} in the next section.}
As the cohomology groups involved in the construction of $\L_X$ are birationally invariant, one can assume that there is a surjective 
morphism $f: X\rightarrow Y$, with $Y$ a compact K\"ahler manifold of dimension $d-k$, such that $f^* : H^1 (Y, \OO_Y) \overset{\simeq}{\longrightarrow} H^1 (X, \OO_X)$ and that there is a map $Y\rightarrow {\rm Alb}(X)$, generically finite onto its image.
Noting that a surjection between compact K\"ahler manifolds $f$ induces injective maps 
$$f^* : H^i (Y, \OO_Y) \hookrightarrow H^i (X, \OO_X) {\rm~ for~ all~} i,$$ 
we obtain an inclusion of complexes $\L_Y \hookrightarrow \L_X$. Now the rightmost 
term in $\L_Y$ is $S\otimes_{\CC} H^{d-k} (Y, \OO_Y)$. By Serre duality this is non-zero, since $H^0 (Y, \omega_Y) \neq 0$ by virtue of $Y$ being of maximal Albanese dimension. If $0\neq \alpha \in H^{d-k} (Y, \OO_Y)$ then obviously $d_Y (1\otimes \alpha) = 0$, hence also
$d_X (1\otimes f^*\alpha) = 0$, where $d_Y$ and $d_X$ are the differentials of $\L_Y$ and $\L_X$ respectively. But $f^*\alpha \neq 0$, and $1\otimes f^*\alpha$ cannot be in the image of $d_X$, since it has degree $0$ with respect to the $S$-grading.
\end{proof}

\begin{remark}[Converse to the Generic Vanishing theorem \cite{gl1}]
Arguing as in the proof of Proposition \ref{BGG.Exactness.Prop}, the conclusion of Proposition \ref{BGG.Non.Exactness.Prop} is equivalent to the fact that $(R^{d-k} {p_2}_* \P)_0 \neq  0$. According to \cite{PP2} Theorem 2.2, this is in turn equivalent to the fact that around the origin, the cohomological support loci of $\omega_X$ (see Proposition \ref{basic} below) 
satisfy
$${\rm codim}_0 V^i (\omega_X) \ \ge \ i - k \ {\rm~ for~ all ~} i> 0.$$
Therefore this last condition becomes equivalent to $\dim a(X) \le k$, which is a converse to the main result of \cite{gl1} 
(in a strong sense, as it has the interesting consequence that the behavior of the $V^i (\omega_X)$ is dictated by their behavior 
around the origin).
\end{remark}

Finally, we record a criterion to guarantee that the BGG sheaf $\FF$ on the projective space $\PP$  is locally free.  Recall that an \textit{irregular fibration} of $X$ is a surjective morphism $f : X \lra Y$ with connected fibres from $X$ onto a normal variety $Y$ with $0 < \dim Y < \dim X$ having the property that a smooth model of $Y$ has maximal Albanese dimension.  
 \begin{proposition}\label{basic}
 \begin{enumerate}
 \item[(i).]
 If $X$ has maximal Albanese dimension, then $\underline{\bf L}_X$ is a resolution of $\FF$. 
 \vskip 3pt
 \item[(ii).] Suppose that  $0 \in \Picc^0(X)$ is an isolated point of the cohomological support loci
 \[  
V^i(\omega_X) \ =_{\tn{def}} \big \{ \alpha \in \Picc^0(X) \, \mid \, \HH{i}{X}{\omega_X \otimes \alpha} \ne 0 \, \big \}
\]  for every  $i >0$.
 Then $\FF$ is a vector bundle on $\PP$, with
 $ \tn{rk}(\FF)= \chi (\omega_X)$. \vskip 3pt
\item[(iii).] The hypothesis of \tn{(}ii\tn{)} holds in particular  if $X$ does not carry any irregular fibrations.  \end{enumerate}
 \end{proposition}
\begin{proof}
The first statement is the case $k = 0$ of Proposition \ref{BGG.Exactness.Prop}, and (iii) follows from   \cite{gl2}, Theorem 0.1. In general, $V^i(\omega_X)$ contains the support of the direct image $R^i p_{2*} \mathcal{P}$. Hence if  the $V^i(\omega_X)$ are finite for $i > 0$,  then  the corresponding direct images   are supported at only finitely many points, and this implies that the vector bundle maps appearing in $\underline{\bf L}_X$ are everywhere of constant rank. (Compare \cite{el} or \cite{hp}, Proposition 2.11.) Thus $\FF$ is locally free, and its rank is computed from   $\underline{\mathbf{L}}_X$. \end{proof}

\begin{remark} \label{SupportLociFinite}
We note for later reference that the main result of \cite{gl2} asserts more generally that if $X$ doesn't admit any irregular fibrations, then in fact $V^i(\omega_X)$ is finite for every $i > 0$. 
\end{remark}

\section{BGG and the canonical cohomology module}

In this section we apply main technical result Theorem \ref{GV_cons} to study the regularity of the canonical module $Q_X$. We also discuss a variant involving a twisted BGG complex. 
    
We start by briefly recalling from \cite{efs} and \cite{eisenbud} some basic facts concerning the BGG correspondence.
Let $V$ be a $q$-dimensional complex vector space,\footnote{The BGG correspondence works over any field, but in the interests of unity we stick throughout to $\CC$.} and let 
$E = \oplus_{i=0}^d \bigwedge^i V$ be the exterior algebra over $V$. Denote by 
$W = V^\vee$ be the dual vector space, and by $S = {\rm Sym}(W)$ the symmetric algebra over $W$.       Elements of $W$ are taken to have degree $1$, while those in $V$ have degree $-1$.

Consider now a finitely generated graded module  $P = \oplus_{i=0}^dP_i$  over $E$. The \emph{dual} over $E$ of the module $P$ is defined to be the $E$-module 
\[ Q\ = \  \widehat{P} \ = \ \oplus_{j=0}^d P_{-j}^\vee \] (so that positive degrees are switched to negative ones and vice versa).  The basic idea of the BGG correspondence is that the properties of $Q$ as an $E$-module are controlled by a linear complex of $S$-modules constructed  from $P$. Specifically, one considers the complex $\L(P)$ given by
$$\ldots \lra S \otimes_{\CC} P_{j+1} \lra S\otimes_{\CC} P_j \longrightarrow S\otimes_{\CC} P_{j-1}\lra \ldots$$ 
with differential induced by 
$$s\otimes p \mapsto \sum_i x_i s \otimes e_i p,$$ 
where $x_i \in W$ and $e_i \in V$ are dual bases. 
We refer to \cite{efs} or \cite{eisenbud} for a dictionary linking $\L(P)$ and $Q$. 

It is natural to consider a notion of regularity for $E$-modules analogous to the theory of Castelnuovo-Mumford regularity for finitely generated $S$-modules.  We limit ourselves here to  modules concentrated in non-positive degrees. 
\begin{definition} [Regularity] \label{regular} 
A finitely generated graded $E$-module $Q$ with no component of positive degree is called \emph{$m$-regular} if it is generated in degrees $0$ up to $-m$, and if its minimal free resolution has at most $m+1$ 
linear strands. Equivalently, $Q$   is $m$-regular if 
and only if 
\[ {\rm Tor}_i^E (Q, \CC)_{-i-j} \ = \ 0\]  for all $i \ge 0$ and all $j \ge m+1$.
\end{definition}
 \noi As an immediate application of the results of Eisenbud-Fl\o ystad-Schreyer, one has  the following addendum 
to \cite{efs} Corollary 2.5 (cf. also \cite{eisenbud} Theorems 7.7, 7.8), suggested to us by F.-O. Schreyer.

\begin{proposition}\label{regularity}
Let $P$ be a  finitely generated graded module over $E$ with no component of negative degree, say $P = \oplus_{i=0}^d P_i$.   Then $Q = \widehat{P}$ is $m$-regular if and only if $\L (P)$ is exact at the first $d-m$ steps from the left, i.e. if and only if the sequence
\[
0 \lra S \otimes_{\CC} P_{d} \lra S\otimes_{\CC} P_{d-1}\longrightarrow \ldots \lra S \otimes_{\CC} P_{m}\]
of $S$-modules is exact.  \qed
\end{proposition}

We propose to apply this machine to the canonical cohomology module.
As before, 
let $X$ be a compact K\"ahler manifold of dimension $d$, and $\alb_X: X \lra  {\rm Alb}(X)$ 
its Albanese map. Set 
\[ V= \HH{1}{X}{\OO_X} \ \ , \ \ E = \Lambda^* V \ \ ,  \ \ W  = V^\vee \ \ , \ \ S = \tn{Sym}(W).  \]
We are interested in the graded $E$-modules
\[ P_X \, = \, \underset{i = 0}{\overset{d}{\oplus}} ~ \HH{i}{X}{ \OO_X} \ \ , \ \ 
Q_X \, = \, \underset{i = 0}{\overset{d}{\oplus}}~ H^i (X, \omega_X),
\] 
the $E$-module structure arising from wedge product with elements of $\HH{1}{X}{\OO_X}$. These become dual  modules (thanks to Serre duality) provided that we assign
$\HH{i}{X}{\OO_X}$ degree $d-i$, and $\HH{i}{X}{\omega_X}$ degree $-i$.

According to Proposition \ref{regularity}, the regularity of $Q_X$ is governed by the exactness of the complex $\L(P_X)$ associated to $P_X$. Thus Theorem \ref{Intro.Reg.Thm}
 from the Introduction follows at once from statement (i) of Theorem \ref{GV_cons} in view of the following:
\begin{lemma}
The complex $\L(P_X)$  coincides with the complex $\L_X$ appearing in \eqref{BGG.Complex.Intro}.
\end{lemma}
\begin{proof}
It follows easily from the definitions that the differentials of both complexes are given on the graded piece corresponding to 
any $p\ge 0$ and $i\ge 0$ by 
$$S^{p-1} W\otimes H^{i-1} (X, \OO_X) \longrightarrow S^p W\otimes H^i (X, \OO_X)$$
induced by the cup-product map $V\otimes H^{i-1} (X, \OO_X) \rightarrow  H^i (X, \OO_X)$ and the natural map
$S^{p-1} W \rightarrow S^p W \otimes V$.
\end{proof}

\begin{remark}[{Exterior Betti numbers}]
The exterior graded Betti numbers of $Q_X$ are computed as the dimensions of the vector spaces
${\rm Tor}_i^E (Q, \CC)_{-i-j}$. When $X$ is of maximal Albanese dimension and $q(X) > \dim X$,  Theorem \ref{Intro.Reg.Thm} implies that these vanish for $i \ge 0$ and $j \ge 1$, and the $i$-th Betti number 
in the linear resolution of $Q_X$ is 
$$b_i \ =\ \dim_{\CC} {\rm Tor}_i^E (Q, \CC)_{-i} \ = \ h^0 (\PP, \F(i))$$ 
where $\F$ is the BGG sheaf defined in \eqref{BGG.Sheaf.Eqn.Intro} in the Introduction. (The last equality follows from general machinery,  
cf. \cite{eisenbud} Theorem 7.8.) On the other hand $\F$ is $0$-regular in the sense of Castelnuovo-Mumford  by virtue of having a linear resolution, so the higher cohomology of its nonnegative twists 
vanishes. Hence $b_i = \chi (\PP, \F(i))$, i.e. the  exterior Betti numbers are computed by the Hilbert 
polynomial of $\F$. 
\end{remark}

\begin{remark} [Alternative proof of Proposition \label{Streamlined.Non-Exactness.Rmk}  \ref{BGG.Non.Exactness.Prop}] \label{Streamlined.Non.Exactness}
One can use the BGG correspondence to deduce the non-exactness statement of Proposition \ref{BGG.Non.Exactness.Prop} from a theorem of Koll\'ar and its extensions. In fact, in view of Proposition \ref{regularity}, it is equivalent to prove that $Q_X$ is not 
$(k-1)$-regular. To this end, observe that the main result of \cite{kollar2} (extended in \cite{saito} and \cite{takegoshi} to the K\"ahler 
setting) gives the splitting 
$\RR a_* \omega_X \cong  \oplus_{j = 0}^k \, R^ja_* \omega_X\, [-j]$ 
in the derived category of $A$. This implies that $Q_X$ can be expressed as a direct sum
\[
Q_X \ = \ \oplus_{j=0}^k\,   Q^j[j] \ \ ,  \ \ \text{with} \ \  Q^j  \ = \ \HH{*}{A}{R^ja_* \omega_X}.\]
Moreover this is a decomposition of $E$-modules: $E$ acts on $\HH{*}{A}{R^ja_* \omega_X}$ via cup product through the identification $\HH{1}{X}{\OO_X} = \HH{1}{A}{\OO_A}$, and we again consider $\HH{i}{A}{R^ja_* \omega_X}$ to live in degree $-i$. We  claim next that $Q^k \ne 0$. In fact, each of the $R^ja_* \omega_X$ is supported on the $(d-k)$-dimensional Albanese image of $X$, and hence has vanishing cohomology in degrees $> d-k$. Therefore $\HH{d}{X}{\omega_X} = \HH{d-k}{A}{R^{k}a_* \omega_X} $, which shows that $Q^k \ne 0$.  On the other hand,  $Q^k [k]$ is concentrated in degrees $\le -k$, and therefore $Q_X$ must have  generators in degrees $\le -k$. 
\end{remark}

\begin{remark} Keeping the notation of the previous Remark, the authors and C. Schnell have shown that each of the modules $Q^j$ just introduced is $0$-regular.
Thus the minimal $E$-resolution of $Q_X$ splits into the direct sum of the (shifted)-linear resolutions of the modules 
$Q^j [j]$. Details will appear in a forthcoming note. \end{remark}

Finally, we discuss briefly a variant involving twisted modules. Fix an element $\alpha \in \Picc^0(X)$, and set
\[
P_\alpha \ = \ \oplus\,  \HH{i}{X}{\alpha} \ \ , \ \ Q_{\alpha} \ = \ \oplus\,  \HH{j}{X}{\omega_X \otimes \alpha^{-1}}.
\]
With the analogous grading conventions as above, these are again dual modules over the exterior algebra $E$. Letting 
\[  t \ = \ t(\alpha) \ = \ \max \{ \, i \mid \HH{i}{X}{ \omega_X \otimes \alpha^{-1}} \ne 0\,  \}, \]
the BGG complexes $\mathbf{L}(P_\alpha)$ and $\underline{\mathbf{L}}(P_\alpha)$ for $P_\alpha$ take the form 
\begin{equation}
0\rightarrow S\otimes H^{d-t} (X, \alpha) \rightarrow  S\otimes H^{d-t+1} (X, \alpha)\rightarrow \ldots \rightarrow S \otimes H^{d-1}(X, \alpha)\rightarrow S \otimes H^d (X, \alpha)\rightarrow 0 \notag
\end{equation}  and\begin{multline}
 0\rightarrow \OO_{\PP }(-t)\otimes H^{d-t} (X, \alpha) \rightarrow  \OO_{\PP } (-t+1)\otimes H^{d-t+1} (X, \alpha)\rightarrow \ldots \rightarrow  \\ \rightarrow \OO_{\PP} (-1) \otimes H^{d-1}(X, \alpha)\rightarrow \OO_{\PP } \otimes H^d (X, \alpha)\rightarrow 0.   \notag
\end{multline}
Writing as before $k = k(X)$ for the generic fibre dimension of the Albanese map, it follows as above from \cite{hacon}, \cite{pareschi}, \cite{PP2} that $\mathbf{L}(P_\alpha)$ is exact at the first $d -t - k$ steps from the left. Hence:
\begin{variant}\label{twisted}
The $E$-module $Q_\alpha$ is $k$-regular. \qed
\end{variant}
\noi We will return to the sheafified complex $\underline{\mathbf{L}}(P_\alpha)$ later.

\begin{remark} [Holomorphic forms] One can also extend aspects of the present discussion to $E$-modules associated to other bundles of holomorphic forms. This is worked out by Lombardi \cite{lombardi}, who gives some interesting applications: see Remark \ref{lombardi}.  
\end{remark}


\section{Inequalities for Numerical Invariants}

In this section, we use the BGG-sheaf $\FF$ to study numerical invariants of a compact K\"ahler manifold. The exposition proceeds in three parts. We begin by establishing Theorem \ref{Intro.Chern.Ineq.Thm} from the Introduction. In the remaining two subsections we discuss examples, applications and variants. 

\subsection*{Inequalities from the BGG bundle}
As before, let $X$ be a compact K\"ahler manifold of dimension $d$, and write
\[
p_g \ = \ \hh{0}{X}{\omega_X} \ \ , \ \ \chi \ = \ \chi\big(X, \omega_X\big) \ \ , \ \ q \ = \ \hh{1}{X}{\OO_X} \ \ , \ \ n = q-1.\]
Denote by $\PP = \PPP \big(\HH{1}{X}{\OO_X}\big)$, so that $\PP$ is a projective space of dimension $n = q-1$, and by $\FF =\FF_X$ the BGG-sheaf on $\PP$  introduced in \eqref{BGG.Sheaf.Eqn.Intro}. 
Once one knows that $\FF$ is locally free, more or less elementary arguments with vector bundles on projective space yield inequalities for numerical invariants. 
As in the Introduction, for $1 \le i \le q-1$ define $\gamma_i = \gamma_i(X)$ to be the coefficient of $t^i$ in the formal power series 
$$\gamma  (X; t) \ =_{\text{def}} \  \prod_{j = 1}^{d} (1 - jt)^{(-1)^j  h^{d, j}} \ \in \ \ZZ[[t]],$$
with $h^{i,j} = h^{i,j}(X)$. The following statement recapitulates Theorem \ref{Intro.Chern.Ineq.Thm} from the Introduction under a slightly weaker hypothesis
(cf. Proposition \ref{basic}).

\begin{theorem} \label{Stated.Thm.Section.3}
Assume that $0 \in \Picc^0(X)$ is an isolated point of $V^i(\omega_X)$ for every $i > 0$.Then  
\begin{enumerate} \item[(i).] Any Schur polynomial of weight $\le q-1$ in the $\gamma_i$ is non-negative. In particular \[ \gamma_i(X) \ \ge \ 0\] for every $1 \le i \le q-1$. 
\vskip 4pt
\item[(ii).] If $i$ is any index with $\chi (\omega_X) < i < q$, then $\gamma_i(X) = 0$. 
\vskip 4pt
\item[(iii).] One has $\chi (\omega_X) \ge q  - d$. \end{enumerate}
\end{theorem}
\begin{proof}
Thanks to Proposition \ref{basic}, the hypothesis guarantees that $\FF$ is locally free, and has a linear resolution:
\begin{multline}  \label{BGG.Lin.Resoln}
 0\lra \OO_{\PP}(-d)\otimes \HH{0}{X}{\OO_X} \lra  \OO_{\PP} (-d+1)\otimes \HH{1}{X}{\OO_X}\lra  \ldots \\ \ldots \lra \OO_{\PP} (-1) \otimes \HH{d-1}{X}{\OO_X} \lra \OO_{\PP} \otimes \HH{d}{X}{\OO_X} \lra \FF \lra 0. 
  \end{multline}
  Identifying as usual cohomology classes on $\PP^n$ with integers,    $\gamma(X;t)$ is then just the Chern polynomial of $\FF$. On the other hand, as   $\FF$   is globally generated,  the Chern classes $c_i(\FF)$ --  as well as the Schur polynomials in these -- and represented by effective cycles. Thus \[
\gamma_i(X) \ = \ \deg c_i(\FF) \ \ge \ 0.
\]
The second statement follows from the fact that $c_i(\FF) = 0$ for $i > \rk(\FF)$. 

Turning to (iii), we may assume that $q > d$ since in any event $\chi \ge 0$ by generic vanishing. If $q-d = 1$, then the issue is to show that $\chi = \rk(\FF) \ge 1$, or equivalently that $\FF \ne 0$. But this is clear, since there are no non-trivial exact complexes of length $n$ on $\PP^n$ whose terms are sums of line bundles of the same degree. So we may suppose finally that $q-1 = n >d$. The quickest argument is note that chasing through \eqref{BGG.Lin.Resoln}
 implies that $\FF$ and its twists have vanishing cohomology in degrees $0 < j < n - d - 1$. But if $\chi \le n-d$ this means by a result of Evans--Griffith, \cite{eg} Theorem 2.4, that $\FF$ is a direct sum of line bundles, which as before is impossible. (See \cite{positivity}, Example 7.3.10, for a quick proof of this splitting criterion due to Ein, based on Castelnuovo-Mumford regularity and vanishing theorems for vector bundles.) 

For a more direct argument in the case at hand that avoids Evans--Griffith, let $s \in \HH{0}{\PP}{\FF}$ be a general section, and let $Z = \tn{Zeroes}(s)$.  We may suppose that $Z$ is non-empty -- or else we could construct a vector bundle $\FF^\prime$ of smaller rank having a linear resolution as in \eqref{BGG.Lin.Resoln} -- and smooth of dimension $n - \chi$. Splicing together the sequence \eqref{BGG.Lin.Resoln} and the Koszul complex determined by $s$, we arrive at a long exact sequence having the shape:
\begin{multline}
0 \lra \OO_{\PP}(-d) \lra \oplus\, \OO_{\PP}(-d + 1) \lra \ldots \lra  \oplus \, \OO_{\PP}(-1) \lra \\ \lra \oplus\,  \OO_{\PP}  \lra \Lambda^2 \FF \lra \ldots \lra \Lambda^{\chi -1}\FF \lra \OO_{\PP}(c_1) \lra \OO_Z(c_1) \lra 0, 
\tag{*}
\end{multline}
where $c_1 = c_1(\FF)$. Observe that $\omega_Z = \OO_Z(c_1-n-1)$ by adjunction. Since $\FF$ is globally generated, a variant of the Le Potier vanishing theorem\footnote{
The statement we use is that if $\mathcal{E}$ is a nef vector bundle of rank $e$ on a smooth projective variety $V$ of dimension $n$, then 
\[ \HH{i}{V}{\Lambda^a \mathcal{E} \otimes \omega_V \otimes L} = 0
\] for $i > e-a$ and any ample line bundle $L$.
In fact, it is equivalent to show that $\HH{j}{V}{\Lambda^a \mathcal{E}^* \otimes L^*} = 0$ for $j < n + a - e$. For this,  after passing to a suitable branched covering as in \cite{positivity}, proof of Theorem 4.2.1, we may assume that $L = M^{\otimes a}$, in which case the statement follows from Le Potier vanishing in its usual form: see \cite{positivity}, Theorem 7.3.6. 
}
yields that 
\[ \HH{i}{\PP}{\Lambda^a \FF \otimes \omega_{\PP} (\ell)} = 0 \ \ \text{for} \ \ i> \chi-a\ , \ \ell > 0. \]
Now twist through in (*) by $\OO_{\PP}(d-n-1)$. Chasing through the resulting long exact sequence, one finds that $\HH{n-d-(\chi-1)}{Z}{\omega_Z(d)} \ne 0$. But if $\chi \le n-d$, this contradicts Kodaira vanishing on $Z$. 
\end{proof}

\begin{remark} [{Evans--Griffith Theorem}]\label{irrat}   
A somewhat more general form of (iii) appears in \cite{PP2} and can be deduced here as well:
applying Variant \ref{gv_index} and using more carefully the results of \cite{gl2}, one can 
assume only that there are no irregular fibrations $f : X \lra Y$ such that $Y$ is generically finite onto a \emph{proper} subvariety of a complex torus (i.e. $X$ has no higher irrational pencils in Catanese's 
terminology \cite{catanese}.)
The argument in \cite{PP2} involved applying the Evans--Griffith syzygy theorem to the Fourier--Mukai transform of the Poincar\'e bundle on $X \times \Picc^0(X)$. The possibility  mentioned in the previous proof of applying the Evans--Griffith--Ein splitting criterion to the BGG bundle $\FF$ is related but substantialy quicker. As we have just seen the additional information that $\FF$ admits a linear resolution allows one to bypass Evans--Griffith altogether, although as in Ein's proof we still use vanishing theorems for vector bundles. 
\end{remark}

\subsection*{Hodge-number inequalities} Here we give some examples and variants of the inequalities appearing in the first assertions of Theorem  \ref{Stated.Thm.Section.3}. To put things in context, we start with an extended remark on what can be deduced from previous 
work of various authors.

\begin{remark}[Inequalities deduced from \cite{catanese} and \cite{cp}]\label{catanese}
Catanese \cite{catanese} shows that if a compact K\"ahler manifold $X$ admits no irregular fibrations, then  the natural 
maps 
\begin{equation} \label{wedge.forms.map}
\phi_k: \bigwedge^k H^i (X, \OO_X) \longrightarrow H^k (X, \OO_X) \end{equation}
are injective on primitive forms $\omega_1 \wedge\ldots \wedge \omega_k$, for $k \le \dim X$. Since these correspond to the Pl\"ucker embedding of the Grassmannian ${\bf G}(k, V)$, one obtains the bounds $h^{0 , k}(X) \ge k(q(X) - k) + 1$. This includes the classical
$h^{0,2} (X) \ge 2q(X) - 3$.  However, still based on Catanese's results, one of the referees points out 
a nice argument that provides the even stronger inequality:
\begin{equation} 
\label{referee} h^{0,2}(X) \ \ge 4q(X) - 10 \ \end{equation}
provided that $\dim X \ge 3$. We thank the referee for  allowing us to include the proof here.

Assume then that $X$ is a compact K\"ahler manifold of dimension $\ge 3$ with no irregular fibrations. By
\cite{catanese}, for any independent $1$-forms $\omega_1$, $\omega_2$, $\omega_3$ on $X$
one has $\omega_1\wedge \omega_2 \wedge \omega_3 \neq 0$ in $H^0 (X, \Omega_X^3)$. Accordingly, 
for any independent $\omega_1, \omega_2, \omega_3, \omega_4 \in H^0 (X, \Omega_X^1)$ one has
$$\omega_1 \wedge \omega_2 \neq \omega_3 \wedge \omega_4 \in H^0 (X, \Omega_X^2).$$
Writing $W = ~<\omega_1, \omega_2, \omega_3, \omega_4> ~\subseteq H^0 (X, \Omega_X^1)$, we then have that the natural map
$$ \bigwedge^2 W \rightarrow H^0 (X, \Omega_X^2)$$
is injective (as the secant variety of the Grassmannian ${\bf G}(2,4)$ fills up the ambient $\PP^5$ of the Pl\"ucker embedding). 
Denote now by $\E$ the tautological sub-bundle on the Grassmannian of subspaces
${\bf G} = {\bf G} (4, H^0 (X, \Omega_X^1) )$. We obtain a morphism 
$$ {\bf R} := \PPP (\bigwedge^2 \E) \overset{\varphi}{\longrightarrow} \PPP (H^0 (X, \Omega_X^2)) $$
which is given by sections of the line bundle $\OO_{\PP}(1)$. As this line bundle is big,\footnote{
This is equivalent to the assertion that the map $ \Phi : {\bf R} \lra \PPP (\bigwedge^2 H^0 (X, \Omega_X^1) )$ is generically finite over its image. But  it is elementary to see that the image of $\Phi$ coincides with the secant variety of the Plucker embedding of ${\bf G}^\pr = {\bf G}(2, H^0 (X, \Omega_X^1))$, and this secant variety is well-known to have dimension $= 2 \dim {\bf G}^\pr + 1 - 4 = 4q - 11 = \dim {\bf R}$.
}
we have
${\rm dim}~{\rm Im}(\varphi) = {\rm dim}~\RR$, and \eqref{referee} follows.

Finally we note that Causin and Pirola \cite{cp} provide a more refined study in the case $k = 2$ of the homomorphism appearing in \eqref {wedge.forms.map}. Among other things, 
for $q(X) \le 2\dim X -1$ they show that $\phi_2$ is injective, so $h^{0,2}(X) \ge {q(X)\choose 2}$. 
Thus for threefolds with $q(X) = 5$ and no irregular fibrations, their result coincides with (\ref{referee}).\footnote{The results of Catanese and Causin-Pirola actually only assume the absence of higher irrational pencils, cf. Remark \ref{irrat}.}
\end{remark}

\begin{example}[Theorem \ref{Stated.Thm.Section.3}  in small dimensions]
We unwind a few of the inequalitites  in  statement (i) of Theorem \ref{Stated.Thm.Section.3}. We assume that $X$ has dimension $d$, and that it carries no irregular fibrations. For compactness, write $h^{0,j} = h^{0,j}(X)$ and $q= q(X)$. To begin with, the condition $\gamma_1 \ge 0$ gives a linear inequality among $q, h^{0,2}, \ldots, h^{0,d-1}$. In small dimensions this becomes:
\begin{equation} \label{Linear.Ineqs}
\begin{aligned}
h^{0,2}  \ &\ge \ -3 + 2q \ \ &\text{ when $d = 3$;} \\
h^{0,3}  \ &\ge \ 4 - 3q + 2 h^{0,2} \ \ &\text{ when $d = 4$;} \\
h^{0,4} \ &\ge \ -5 + 4q- 3 h^{0,2} + 2 h^{0,3} \ \ &\text{ when $d = 5$.}  
\end{aligned}
\end{equation}
Similarly, $\gamma_2$ is a quadratic polynomial in the same invariants, and one may solve $\gamma_2 \ge 0$ to deduce the further  and stronger inequalities:
\begin{equation} \label{Quadr.Ineqs}
\begin{aligned}
h^{0,2}  \ &\ge \ -
\frac{7}{2} + 2q  + \frac{\sqrt{8q -23}}{2}  \ \ &\text{ when $d = 3$;} \\
h^{0,3}  \ &\ge \  \frac{7}{2} - 3q +  2 h^{0,2}  + \frac{\sqrt{ 49 - 24q + 8h^{0,2}}}{2} \ \ &\text{ when $d = 4$;} \\
h^{0,4} \ &\ge \ -\frac{11}{2} + 4q - 3h^{0,2} + 2 h^{0,3} +\frac{ \sqrt{71 + 48q -24h^{0,2} + 8 h^{0,3}}}{2}   \ \ &\text{ when $d = 5$,}  
\end{aligned}
\end{equation}
where in the last case we assume that the expression under the square root is non-negative. (This is automatic when $d = 3$ since $q \ge 3$, and when $d = 4$ thanks to \eqref{referee}.) Note that equality holds in \eqref{Linear.Ineqs} when $X$ is an abelian variety (in which case $\FF = 0$). Similarly, when $X$ is a theta divisor in an abelian variety of dimension $d + 1$,  then $\rk \FF = 1$, so equality holds in each of the three instances of \eqref{Quadr.Ineqs}. 
\end{example}

When $d = 3$ or $d = 4$, we can combine the various inequalities in play to get an asymptotic statement:
\begin{corollary}\label{hodgethreefold}
$($i$)$. If $X$ is an irregular compact K\"ahler threefold with no irregular fibration, then \[ h^{0,3}(X) \ \ge \ h^{0,2}(X) - 2, \] so asymptotically $$ h^{0,2} (X) \ \succeq \ 4q(X) \  \ {\rm ~and~} \ \ h^{0,3}(X) \succeq 4q(X).$$
\vskip 4pt
\noindent
$($ii$)$. If $X$ is an irregular compact K\"ahler fourfold with no irregular fibration, then asymptotically 
$$ h^{0,2} (X) \ \succeq \ 4q(X)\ \ , \  \   \ h^{0,3}(X) \ \succeq \  5q(X) + \sqrt{2q(X)}, 
\ \ {\rm ~and~}\  \ h^{0,4}(X)\ \succeq  \ 3q(X)  + \sqrt{2q(X)}.$$
\end{corollary}
\begin{proof}
(i). The first inequality is equivalent to the statement $\chi(\omega_X) \ge q - 3$ from statement (iii) of Theorem \ref{Stated.Thm.Section.3}, and the other inequalities follow from this and \eqref{referee}.

(ii). The inequality $\chi(\omega_X) \ge q - 4$ is equivalent to
\[   h^{0,4}(X) \ \ge \ (2q - 5) + \big( h^{0,3}(X) - h^{0,2}(X) \big). \]
The statement then follows by using \eqref{Quadr.Ineqs} to bound $h^{0,3} - h^{0,2}$, and invoking the inequality $h^{0,2} \succeq 4q$ coming from \eqref{referee}. \end{proof}

Similar arguments lead to a new inequality for $h^{1,1}$ on a  surface.
Specifically, let $X$ be a   compact K\" ahler surface with no non-constant morphism to a curve of genus at least $2$. The classical Castelnuovo-de Franchis inequality asserts  that $h^{0,2}(X) \ge 2q (X) - 3$. A related result based on 
the Castelnuovo-de Franchis Lemma and linear algebra also bounds $h^{1,1}$ in terms of the 
irregularity: it is shown in \cite{bhpv}, IV.5.4,  that $h^{1,1}(X) \ge 2 q(X) - 1$. The methods of the present paper yield  a strengthening of this:\footnote{We remark that in the case of surfaces of general type 
without irrational pencils the inequality $h^{1,1}(X) \ge 3q(X) - 2$ could also be obtained by 
combining the Castelnuovo-de Franchis inequality with the deep Bogomolov-Miyaoka-Yau inequality.}
\begin{proposition}\label{surfaces}
If $X$ is a compact K\"ahler surface without irrational pencils, then
$$h^{1,1}(X)\  \ge  \ \left\{ \begin{array}{ll} 3q(X) - 2 & \mbox{if } q(X) \mbox{ is even} \\
								3q(X) - 1 &  \mbox{if } q(X) \mbox{ is odd}   \end{array} \right.$$
\end{proposition}
\begin{proof}
It is well known that given any non-zero one-form   $\omega \in H^0 (X, \Omega_X^1)$ on such a surface $X$,  the map $H^1 (X, \OO_X) \overset{\wedge \omega}{\longrightarrow} H^1 (X, \Omega_X^1)$  obtained by wedging with $\omega$ is injective. On the other hand, this map is naturally dual to the 
map $H^1 (X, \Omega_X^1) \overset{\wedge \omega}{\longrightarrow} H^1 (X, \Omega_X^2)$, via Serre duality. Hence in the  natural complex 
$$0 \longrightarrow H^1 (X, \OO_X) \overset{\wedge \omega}{\longrightarrow} H^1 (X, \Omega_X^1)
\overset{\wedge \omega}{\longrightarrow} H^1 (X, \Omega_X^2)\longrightarrow 0,$$ 
the first map is injective and the second is surjective. Globalizing, we obtain   a monad of vector bundles on $\PP : = \PPP(\HH{0}{X}{\Omega_X^1})$:
$$0 \longrightarrow \OO_{\PP} (-1)^{q} \longrightarrow \OO_{\PP}^{h^{1,1} (X)} 
\overset{\phi}{\longrightarrow} \OO_{\PP}(1)^{q}  \longrightarrow 0.$$
The cohomology $E$ of this monad sits in an exact sequence
$$0 \longrightarrow \OO_{\PP}(-1)^{q}  \longrightarrow K 
\longrightarrow E \longrightarrow 0$$
where $K = {\rm ker}(\phi)$. A direct calculation shows that ${\rm rk} (E) = h^{1,1} (X) - 2q$ and 
$$c_t (E) = \frac{1}{(1-t^2)^{q}} = 1 + qt^2 + {{q + 1}\choose{2}} t^4 + \ldots,$$
with non-zero terms in all even degrees $ \le {\rm dim}~\PP = q -1$. As $c_i (E) = 0$ for $i > {\rm rk} (E)$, this implies that ${\rm rk} (E) \ge q - 2$ if $q$ is even, and ${\rm rk} (E) \ge q - 1$ if $q$ is odd.
\end{proof}

\begin{remark}[Bounds for other Hodge numbers]\label{lombardi}
The techniques of this paper, applied to bundles of holomorphic forms $\Omega_X^p$ as opposed to $\OO_X$, can be used to bound other Hodge numbers for important classes of compact K\"ahler manifolds, where no results in the style of those 
of \cite{catanese} in Remark \ref{catanese}  are available. This is carried out by Lombardi in \cite{lombardi}; for instance, on threefolds whose $1$-forms have at most isolated zeros (e.g. subvarieties of abelian varieties with ample normal bundle), there are 
lower bounds for \emph{all} Hodge numbers in terms of $q(X)$. Besides those mentioned in the proof of Corollary \ref{hodgethreefold}, one has asymptotically
$$ h^{1,1} \ \succeq \ 2q + \sqrt{2q} \quad {\rm and} \quad h^{2,1} 
\ \succeq \ 3q + \sqrt{2q}.$$
\end{remark}

\subsection*{Bounds involving the Euler characteristic}  In this final subsection, we make some remarks surrounding the inequality
\begin{equation} \label{CdF.Ineq}
\chi(X, \omega_X) \ \ge \ q(X) - \dim X 
\end{equation}
established in \cite{PP2} and Theorem \ref{Stated.Thm.Section.3} (iii) for compact K\"ahler manifolds with no irregular fibrations.

Note to begin with that  equality holds in \eqref{CdF.Ineq} when $X$ is birational to a complex torus (in which case $\chi = 0$) or to a theta divisor in a principally polarized abelian variety (in which case $d = q-1$ and $\chi = 1$). It was essentially established by Hacon--Pardini \cite{hp}, \S 4, that in fact these are the only such examples with $\chi \le 1$.\begin{proposition}
\label{small_chi}
Let $X$ be an irregular smooth projective complex variety with no irregular fibrations.
\begin{enumerate}
\item[(i).] If $\chi (\omega_X) = 0$, then $X$ is birational to an abelian variety.

\vskip 4pt
\item[(ii).] If $\chi (\omega_X) = 1 = q(X) - \dim X$, then $X$ is birational to a principal polarization in a \tn{PPAV}. \end{enumerate}
\end{proposition}

Since the statement does not appear explicitly in \cite{hp} we will briefly indicate the proof, but we stress that all the ideas are already present in that paper. 
\begin{proof}[Sketch of Proof]
We again focus on the exact sequence \eqref{BGG.Lin.Resoln}
  of bundles on $\PP = \PP^{q-1}$. Note that in any event $X$ has maximal Albanese dimension, and hence $q \ge d$. 
  If $\chi = 0$, then $\FF = 0$. In this case one reads off from  \eqref{BGG.Lin.Resoln} that  $q = d$ and $P_1(X) = h^{0,d}(X) > 0$. On the other hand,  the assumption of the theorem implies by \cite{gl2} that   $V^i (\omega_X)$ has only isolated points when $i >0$, and since $\chi(\omega_X) = 0$, this implies that $V^0 (\omega_X)$ also consists only of isolated points. But  a result of Ein-Lazarsfeld 
(cf. \cite{chh}, Theorem 1.7) says that a variety of  maximal Albanese dimension with 
$V^0 (\omega_X)$  zero dimensional is birational to its Albanese.
  
Now suppose that $\chi = 1$. Then $\FF$ is a line bundle, and it follows that  \eqref{BGG.Lin.Resoln}  is a twist of the standard Koszul complex, this being  the unique linear complex of length $n+1$ on $\PP^n$ whose outer terms have rank one. In particular $h^{0,d} (X) = q$.
On the other hand we have an injection $H^0 (A, \Omega_A^d) \rightarrow H^0 (X, \Omega_X^d)$. Indeed, since $d = q-1$, if the pullback map were not injective we would have a $d$-wedge of independent holomorphic $1$-forms on $X$ equal to $0$, which by \cite{catanese} Theorem 1.14 would imply the
existence of an irregular fibration.
Now since the two dimensions are equal, the map is in fact an isomorphism. To prove (ii), one can then use a characterization of principal polarizations due to Hacon-Pardini (cf. \cite{hp} Proposition 4.2), extending a criterion of Hacon, which says that the only other thing we need to check is $V^i (\omega_X) = \{0\}$ for all $i >0$.  But this follows from Remarks \ref{SupportLociFinite} and \ref{Non-Triv.Isol.Pts}.
\end{proof}

On the other hand, one expects it to be very rare to find manifolds $X$ with no irregular fibrations for which
$
\chi(X, \omega_X)=q(X) - \dim X   \ge   2.
$
\begin{conjecture}
If $X$ is an irregular compact K\"ahler manifold with no irregular fibrations and $\chi(\omega_X) \ge 2$, 
then \[ \chi(\omega_X) \ > \ q(X) - \dim X
\]
when $q(X)$ is very large compared to 
$\chi(\omega_X)$.
\end{conjecture}
\noi The thinking here is that if equality were to hold in \eqref{CdF.Ineq}, then the BGG-sheaf $\FF$ would provide a non-split vector bundle of small rank on the projective space $\PP$. But these should almost never exist. The fact that $\FF$ admits a linear resolution, and the resulting relations in Theorem \ref{Intro.Chern.Ineq.Thm} (ii), should provide even more constraints.

As an example in this direction, one has  the following, whose proof was shown to us by I. Coand\u a.
\begin{proposition}\label{chi_2}
Let $X$ be a  compact K\"ahler manifold with no irregular fibrations, such that $\chi(\omega_X) = 2$ and $q(X) \ge 5$. 
Then $ q(X) - {\rm dim}~X < 2$.
\end{proposition}

\begin{proof}
Assume for a contradiction that $q(X) - \dim X = 2$, and consider yet again the BGG resolution  \eqref{BGG.Lin.Resoln} of $\FF$. This resolution shows first of all that $\FF$ is $0$-regular in the sense 
of Castelnuovo-Mumford. We next claim that $\HH{1}{\PP}{\FF(-2)} \ne 0$, while $\HH{1}{\PP}{\FF(i)} = 0$ for all $i \neq - 2$. Grant this for the moment. Then the $S$-module $H^1_* (\FF)$ has a non-zero summand in degrees $-2$ and higher. But \cite{mkr} Theorem 1.7 asserts that  a $0$-regular rank $2$ bundle with this property cannot exist when $n \ge 4$. As for the claim, note that \eqref{BGG.Lin.Resoln}  starts 
on the left with a twist of the  Euler sequence, and thus the cokernel of the 
injection $ \OO_{\PP^n}  (-n+1) \rightarrow \OO_{\PP^n}^{n+1} (-n +2)$ is $T_{\PP^n}(-n+1)$. One then finds that 
\[
\HH{1}{\PP^n}{\FF(i)} \ = \ \HH{n-1}{\PP^n}{ T_{\PP^n}(-n+1 + i)},\]
and the assertion  follows from the Bott formula (cf. \cite{oss} p.8--9) and Serre duality.
\end{proof}

\begin{example} [Surfaces and the Tango bundle]
The case of surfaces is particularly amusing from the present point of view.  When $\dim X = 2$,  \eqref {CdF.Ineq}
 is equivalent to the classical Castelnuovo-de Franchis inequality 
 \[ p_g(X) \ \ge \ 2q(X) -3, \] which holds for surfaces with no irrational pencils of genus at least $2$. 
As soon as $q(X) \ge 4$ it has been suggested (cf. e.g. \cite{mlp}) --  and proved by Pirola for $q(X) = 5$ --  that there should be no such surfaces satisfying $p_g (X) = 2q(X) -3$.  If such a surface were to exist, its BGG bundle $\FF$ would have a resolution:
$$0\longrightarrow \OO_{\PP^n} (-2) \longrightarrow \OO_{\PP^n}(-1)^{n+1}  \longrightarrow \OO_{\PP^n}^{2n-1} \longrightarrow \FF \longrightarrow 0$$
where $n = q(X) - 1 \ge 3$.   On the other hand,  for every $n\ge 3$ there does exist  a vector bundle having this shape: it is the dual of the \emph{Tango bundle} (cf. \cite{oss} Ch.I, \S4.3). It would be quite  interesting to decide one way or the other whether one can in fact realize the dual Tango bundle as the BGG bundle of a surface.
\end{example}

Finally, we discuss a strengthening of the inequality  \eqref{CdF.Ineq}, also appearing in \cite{PP2}, which involves the twisted BGG complexes introduced in Variant \ref{twisted}.  

As always, let $X$ be a compact K\"ahler manifold, and fix any point $\alpha \in \Picc^0(X)$. Following \cite{PP2}, one defines the \textit{generic vanishing index} of $\omega_X$ at $0 \in \Picc^0(X)$ to be the integer
\[
\tn{gv}_0(X) \ = \ \min_{i > 0} \,  \{ \codim_{0} \, V^i(\omega_X) - i \}.
\]
 The basic generic vanishing theorems assert that $\tn{gv}_0(X) \ge 0$ when $X$ has maximal Albanese dimension, and if $0$ is an isolated point of $V^i(\omega_X)$ for every $i > 0$ then 
\[
\tn{gv}_0(X) \ = \ q(X) \, - \, \dim(X). 
\]
The following statement, which appeared as Corollary 4.1 in \cite{PP2}, therefore generalizes Theorem \ref{Intro.Chern.Ineq.Thm} (iii).

\begin{variant}\label{gv_index}
Assume that $X$ has maximal Albanese dimension. Then
\[
\chi (X, \omega_X) \ \ge \ \tn{gv}_0(X).
\]
\end{variant}
\begin{proof}[Brief Sketch of Proof]
The origin belongs to all the $V^i(\omega_X)$ (cf.\ \cite{el}, Lemma 1.8), and hence there exist a largest index $t >0$, and an irreducible component $Z \subseteq V^t(X)$, such that $ \tn{gv}_0(X) = \codim Z - t$. Choose a general point $\alpha \in Z$ and consider the twisted BGG complex $\underline{\bf L}(P_\alpha)$, giving a resolution of the indicated sheaf $\FF_\alpha$:
\begin{multline}
 0\lra \OO_{\PP }(-t)\otimes H^{d-t} (X, \alpha) \lra  \OO_{\PP } (-t+1)\otimes H^{d-t+1} (X, \alpha)\lra \ldots   \\ \lra \OO_{\PP} (-1) \otimes H^{d-1}(X, \alpha)\lra \OO_{\PP } \otimes H^d (X, \alpha)\lra \FF_\alpha \lra 0.   \tag{*}
\end{multline}
The sheaf $\FF_{\alpha}$ is typically not locally free. But if one chooses a subspace 
\[ W\ \subseteq \ T_{\alpha} \Picc^0(X)  \ = \ \HH{1}{X}{\OO_X} \]
transverse to the tangent space of $Z$ at $\alpha$,  and restricts (*) to the projectivization $\PP^\pr = \PPP W$ of $W$, then it follows as in \cite{el}, Theorem 1.2, that one gets an exact complex 
\begin{multline}
 0\lra \OO_{\PP^\pr }(-t)\otimes H^{d-t} (X, \alpha) \lra  \OO_{\PP^\pr } (-t+1)\otimes H^{d-t+1} (X, \alpha)\lra \ldots  \\ \lra \OO_{\PP^\pr} (-1) \otimes H^{d-1}(X, \alpha)\lra \OO_{\PP^\pr } \otimes H^d (X, \alpha)\lra \mathcal{G} \lra 0    \tag{**}
\end{multline}
where $\mathcal{G}$ is a vector bundle, of rank $\chi(X, \omega_X)$. 
Note that 
\[
\dim \PP^\pr \ = \ \codim Z - 1 \ = \ \tn{gv}_0(\omega_X) + t - 1.
\]
Now the argument proceeds much as in the proof of Theorem \ref{Stated.Thm.Section.3}, using (**) in place of \eqref{BGG.Lin.Resoln}.
\end{proof}

\begin{remark} [Non-trivial isolated points] \label{Non-Triv.Isol.Pts}
A similar argument gives yet another variant, which was also noted in \cite{PP2}.
\begin{quote}
Suppose that $\alpha \in \Picc^0(X)$ is a point having the property that for every $i > 0$ either $\alpha \not \in V^i(\omega_X)$ or else $\alpha$ is an isolated point of $V^i(\omega_X)$. Assume furthermore  $\HH{p}{X}{\alpha} \ne 0$ for some $p <d$, and let  $ p(\alpha)$ be the least index for which this holds. Then 
\[  \chi(X, \omega_X) \ \ge \ q(X) - \dim X + p(\alpha).\]
\end{quote}
Since evidently $p (\alpha) > 0$ if $\alpha \ne 0$, this means that non-trivial isolated points improve the basic lower bound for the Euler characteristic.
\end{remark}


\section{The BGG Sheaf and Paracanonical Divisors}

In this section we study the geometric meaning of the BGG sheaf $\FF$, proving Theorem \ref{Intro.Normal.Cone.Thm}. As always, $X$ is a compact  K\"ahler manifold of dimension $d$, but we do not exclude the possibility that it carries irregular fibrations. Keeping the notation from the Introduction, $\tn{Div}^{\{ \omega \}}(X)$ denotes the space of divisors on $X$ lying over $\Picc^{\{ \omega \}}(X)$, and $\linser{\omega_X} \subseteq\Picc^{\{ \omega \}}(X)$ is the canonical series.

The first point is to relate the Hilbert scheme of paracanonical divisors 
$\tn{Div}^{\{ \omega \}}(X)$ to a suitable direct image of the Poincar\'e bundle on $X \times \Picc^0(X)$. 
\begin{proposition} \label{Divv}
Let $\mathcal{P}$ denote the normalized Poincar\'e bundle on $X \times \Picc^0(X)$.  Then
\[
\tn{Div}^{\{ \omega \}}(X) \ = \ \PP \big( (-1)^* R^d p_{2*}\mathcal{P}\big) 
\]
as schemes over $\Picc^0(X)$, where $(-1) : \Picc^0(X) \lra \Picc^0(X)$ is the morphism given by  multiplication by $-1$, and $p_1, p_2$ are the projections of $X \times \Picc^0(X)$ onto its factors.  
\end{proposition}
\noi We will provide a formal proof shortly, but for a quick plausibility argument note that if $\alpha \in \Picc^0(X)$ is any point,  then the fibre of $\PP ( (-1)^* R^d p_{2*}\mathcal{P}) $ over $\alpha$  is the projective space of one-dimensional quotients of  $\HH{d}{X}{ \alpha^{-1}} $, which thanks to Serre duality is   identified with the projective space parametrizing divisors in the linear series $\linser{\omega_X \otimes\alpha}$. 
Granting Proposition \ref{Divv} for the time being, we complete the
\begin{proof}[Proof of Theorem \ref{Intro.Normal.Cone.Thm}] It is enough to establish the stated isomorphism after pulling back by the exponential map $\tn{exp} : V =  \HH{1}{X}{\OO_X} \lra \Picc^0(X)$, which is \'etale. Then the results of \cite{gl2} quoted in the proof of Proposition \ref{BGG.Exactness.Prop} imply that $\tn{exp}^*\big( R^d p_{2*} \mathcal{P}\big) $ is isomorphic in a neighborhood of the  origin to the cokernel of map
\[
u : \HH{d-1}{X}{\OO_X} \otimes \OO_{\mathbf V} \lra \HH{d}{X}{\OO_X} \otimes \OO_{{\mathbf V}}
\] 
of sheaves on the affine space  $\mathbf{V} = \CC^q$ arising from the right-most terms of the BGG complex. Note that $u$ is given by a matrix of linear forms, and pulling back by $(-1)$ just  multiplies the entries of this matrix by $-1$. The theorem then reduces to a general statement, appearing in the following lemma, concerning the projectivization of the cokernel of a map of trivial vector bundles on affine space defined by a matrix of linear forms. 
\end{proof}
\begin{lemma}
Let $u$ be an $a \times b$ matrix of linear forms on a vector space $\CC^q$, defining maps 
\[
u: \OO_{\CC^q}^a \lra \OO_{\CC^q}^b \ \ , \ \ \overline{u} : \OO_{\PP^{q-1}}(-1)^a \lra \OO_{\PP^{q-1}}^b,
\]
and set 
$ \mathcal{A} =\tn{coker}(u) \ ,  \ \overline{\mathcal{A}} = \tn{coker} (\overline{u})$.
Consider the subscheme 
\[  \PP(\mathcal{A}) \ \subseteq \ \CC^q \times \PP^{b-1}, \]
whose fibre $F$ over the origin $0 \in \CC^q$ is a copy of $\PP^{b-1}$. Then the projectivized normal cone to $F$ in $\PP(\mathcal{A})$ is identified with $\PP(\overline{\mathcal{A}})$ via the natural projection $\PP(\overline{\mathcal{A}}) \lra \PP^{b-1}$.  \end{lemma}

\begin{proof}
An $a \times b$ matrix $u$ of linear forms on $\CC^q$ gives rise to an $a \times q$ matrix $\overline{w}$ of linear forms on $\PP^{b-1}$ having the property that if $\overline{ \mathcal{B}}$ is the cokernel of the resulting map 
\[  \overline{w} : \OO_{\PP^{b-1}}(-1)^a \lra  \OO_{\PP^{b-1}}^q, \]
then 
$ \PP(\mathcal{A})  \cong  \mathbf{V}(\overline{\mathcal{B}})  $ as subschemes of $ \PP^{b-1} \times \CC^q$.
(In fact,  $\overline{w}$ is constructed so that  $\PP(\mathcal {A})$ and $\mathbf{V}(\overline{\mathcal{B}})$ are defined in $ \PP^{b-1} \times \CC^q$ by the same equations.) Under this identification, the issue is to determine the projectivized normal cone to $\mathbf{V}(\overline{ \mathcal{B}})$ along its zero section; equivalently, one needs to determine the exceptional divisor in the blow-up of  $\mathbf{V}(\overline{ \mathcal{B}})$ along this zero-section. But by virtue of \cite[8.7]{EGAII} this exceptional divisor is $\PP(\overline{ \mathcal{B}})$, and  this in turn coincides  with $\PP(\overline{\mathcal{A}})$ as a subscheme of $\PP^{b-1} \times \PP^{q-1}$. 
\end{proof}

It remains only to give the
\begin{proof}[Proof of Proposition \ref{Divv}] As explained in \cite{kleiman}, Theorem 3.13, there is a unique coherent sheaf $\mathcal{Q}$ on $\Picc^0(X)$ characterized by the property that
\begin{equation}\Shomm{\mathcal{Q}, \mathcal{E}} \ = \ p_{2*}(p_1^*\omega_X \otimes \mathcal{P} \otimes p_2^*\mathcal{E}) \tag{*}\end{equation}
for any sheaf $\mathcal{E}$ on $\Picc^0(X)$, and then $\tn{Div}^{\{ \omega \}}(X)  = \PP(\mathcal{Q})$. So we need to establish that (*) holds with $\mathcal{Q} = (-1)^* R^d p_{2*}\mathcal{P}$. To this end, denote as usual by $\RR\Phi_{\mathcal{P}}(\mathcal{G}) \ = \ \RR p_{2*}(p_1^* \mathcal{G} \otimes \mathcal{P}) $ the Fourier-Mukai transform of a sheaf $\mathcal{G}$ on $X$. By the projection 
formula, one has
\[
\RR p_{2*}(p_1^*\omega_X \otimes \mathcal{P} \otimes p_2^*\mathcal{E})
\ \cong \ \RR\Phi_{\mathcal{P}}(\omega_X)  \overset{\mathbf{L}}{\otimes} \mathcal{E},\]
and we claim that it suffices to prove the derived formula
\begin{equation}
\RR \Phi_{\mathcal{P}}\omega_X \overset{\mathbf{L}}{\otimes} \mathcal{E}  \ \cong \ \RR \Shomm{(-1)^* \RR \Phi_{\mathcal{P}} \OO_X [d], \OO_{\Picc^0(X)}}\overset{\mathbf{L}}{\otimes} \mathcal{E} .\tag{**}
\end{equation}
Indeed, suppose that (**) is known. Now 
\begin{equation}
\RR \Shomm{(-1)^* \RR \Phi_{\mathcal{P}} \OO_X [d], \OO_{\Picc^0(X)}}\overset{\mathbf{L}}{\otimes} \mathcal{E} \ \cong  \ \RR \Shomm{(-1)^* \RR \Phi_{\mathcal{P}} \OO_X [d], \mathcal{E}}, \tag{***}\end{equation}
so the right-hand side of (*) is computed as the $0^{\text{th}}$ cohomology sheaf of the right-hand side in (***). 
 But there is a   spectral sequence 
  \[ E_2^{p,q} \ =  \ \mathcal{E}xt^p\big((-1)^*R^{d-q}p_{2*}\mathcal{P}, \mathcal{E}\big) \Rightarrow
   \ R^{p+q} \Shomm{(-1)^* \RR \Phi_{\mathcal{P}} \OO_X [d], \mathcal{E}}\] 
 with $p \ge 0$ and $q \le 0$.  For degree reasons only $\Shomm{(-1)^* R^d p_{2*}\mathcal{P}, \mathcal{E}} $ contributes to 
 the $0^{\text{th}}$ term, so we get the required identity of sheaves. 

So it remains only to prove (**), for which it suffices to establish that
\[
\RR \Phi_{\mathcal{P}} \omega_X \ \cong  \ \RR \Shomm{(-1)^* \RR \Phi_{\mathcal{P}} \OO_X [d], \OO_{\Picc^0(X)}} \ \cong \ \big( \RR \Phi_{\mathcal{P}^\vee} \OO_X \big)^\vee [-d].\]
But this is precisely the commutation of Grothendieck duality with integral functors (see for instance \cite{PP1} Lemma 2.2\footnote{This is proved in \cite{PP1} in the context of smooth projective varieties, but as indicated in \cite{PP2} 
the same proof works  on complex manifolds, due to the fact that the analogue of Grothendieck duality holds by \cite{rrv}.}).
\end{proof}

\begin{remark}
Note that via the BGG correspondence, one can read off the sheaf $\FF$ just from the structure of $Q_X$ as an $E$-module. Thus Theorem \ref{Intro.Normal.Cone.Thm}
 admits the picturesque corollary that $Q_X$ determines the normal cone to the canonical linear series in $\tn{Div}^{\{ \omega \}}(X)$.
\end{remark}

As an application, we consider the question of whether the canonical series $\linser{\omega_X}$ is an irreducible component of the space 
$\tn{Div}^{\{ \omega \}}(X)$ of paracanonical divisors: following Beauville \cite{beauville}, one says  that $\linser{\omega_X}$ is \textit{exorbitant} if this happens.  
An immediate consequence of Theorem \ref{Intro.Normal.Cone.Thm} is 

\begin{corollary} \label{Exorbitant}
The canonical linear series $\linser{\omega_X}$ is exorbitant if and only if the mapping $\PP(\FF) \lra \PP \big( \HH{d}{X}{\OO_X} \big) =  \linser{\omega_X} 
$ in  \eqref{PF.over.Canon.eqn} fails to be surjective. \qed 
\end{corollary}

Under some additional hypotheses, the criterion in the Corollary can be checked numerically.
An amusing consequence of this is that in the setting of Theorem \ref{Stated.Thm.Section.3}, the exorbitance of the canonical series actually depends only on the Hodge numbers of $X$. In fact:
\begin{proposition} \label{Exorbitant.Numerical}
Assume that the hypotheses of  Theorem \ref{Intro.Chern.Ineq.Thm} $($or, more generally, Theorem \ref{Stated.Thm.Section.3}$)$ are satisfied, and that 
\begin{equation}
p_g - \chi \  \le \ q -1. \tag{*}
\end{equation}
Then $\linser{\omega_X}$ is exorbitant if and only if the codimension $(p_g - \chi)$ Segre number of $\FF^\vee$ vanishes, i.e. if and only if:
\[
s_{1^{\times (p_g - \chi)}}\big( \gamma_1, \ldots, \gamma_{q-1}\big) \ = \ 0,
\]
where the quantity in question indicates the Schur function associated to the partition $(1, \ldots, 1)$ \tn{(}$p_g - \chi \tn{)}$ times. 
\end{proposition}
\noi Observe that if $\chi > 0$ then  $\tn{Div}^{\{ \omega \}}(X)$ has a unique component of dimension $q + \chi - 1$ dominating $\Picc^0(X)$, so if (*) fails in this case, then necessarily $\linser{\omega_X}$ is exorbitant. 

\begin{proof}[Proof of Proposition \ref{Exorbitant.Numerical}] According to Corollary \ref{Exorbitant}, $\linser{\omega_X}$ is exorbitant if and only if the natural mapping
\[
\PP(\FF) \lra \PP \big( \HH{d}{X}{\OO_X} \big) \, = \, \PP^{p_g -1}
\] 
fails to be surjective. But the Segre number in question computes the degree in $\PP^{q-1}$ of the preimage of a general point in the target, and the statement follows. 
\end{proof}

\begin{example}
Suppose that $X$ is a surface without irrational pencils.  Then $p_g -\chi = q- 1$, and the BGG complex takes the form
\[
0 \lra \OO_{\PP}(-2) \lra \OO_{\PP}(-1)^q\lra \OO_\PP^{p_g} \lra \FF \lra 0.
\]
In this case the Segre number appearing in Proposition \ref{Exorbitant.Numerical} is the coefficient of $t^{q-1}$ in $(1+t)^q/(1+2t)$, and this is $=0$ if $q$ is even, and $=1$ if $q$ is odd. Thus $\linser{\omega_X}$ is exorbitant if and only if $q$ is even, a fact observed by Beauville in \cite{beauville}, \S4. 
\end{example}

 

\providecommand{\bysame}{\leavevmode\hbox
to3em{\hrulefill}\thinspace}


\begin{thebibliography}{EMS}

\bibitem[BHPV]{bhpv}
W. Barth, K. Hulek, C. Peters and A. Van de Ven, \emph{Compact complex surfaces}, Springer 2004.

\bibitem[Be]{beauville}
A. Beauville, {Annulation du $H^1$ et syst\`emes paracanoniques sur
les surfaces}, J. Reine Angew. Math. \textbf{388} (1988), 149--157.

\bibitem[BGG]{bgg}
I. N. Bernstein, I. M. Gel'fand and S. I. Gel'fand,  {Algebraic vector bundles on $P\sp{n}$ and problems of linear algebra},  Funktsional. Anal. i Prilozhen. \textbf{12} (1978), no. 3, 66--67.

\bibitem[Ca1]{catanese}
F. Catanese, {Moduli and classification of irregular Kaehler
manifolds (and algebraic varieties) with Albanese general type
fibrations}, Invent. Math. \textbf{104} (1991), 389--407. 

\bibitem[Ca2]{Catan}
F. Catanese, {From Abel's heritage: transcendental objects in algebraic geometry
   and their algebraization}, 
   in{ \it{The legacy of Niels Henrik Abel}}, Springer, 2004, 349--394. 
    
\bibitem[CP]{cp}
A. Causin and G. P. Pirola, {Hermitian matrices and cohomology of K\"ahler varieties}, 
Manuscripta Math. \textbf{121} (2006), 157--168.

\bibitem[ChH]{chh}
J. A. Chen and Ch. Hacon, {Characterization of abelian varieties}, Invent. Math. \textbf{143} (2001), no. 2, 435--447.

\bibitem[CH]{ch}
H. Clemens and Ch. Hacon, {Deformations of the trivial line bundle and vanishing theorems}, 
 Amer. J. Math. \textbf{124} (2002), no. 4, 769--815.

\bibitem[Co]{coanda}
I. Coand\u a, {On the Bernstein-Gel'fand-Gel'fand correspondence and a result of Eisenbud, Fl¿ystad, and Schreyer}, J. Math. Kyoto Univ. \textbf{43} (2003), no. 2, 429--439.

\bibitem[Ei]{ein}
L. Ein, {An analogue of Max Noether's theorem}, Duke Math. J. \textbf{52} (1985), no.3, 689--706.

\bibitem[EL]{el}
L. Ein and R. Lazarsfeld, {Singularities of theta divisors and the
birational geometry of irregular varieties}, J. Amer. Math. Soc.
\textbf{10} (1997), 243--258.

\bibitem[Eis]{eisenbud}
D. Eisenbud, \emph{The geometry of syzygies}, Springer 2005.

\bibitem[EFS]{efs}
D. Eisenbud, G. Fl\o ystad and F.-O. Schreyer, {Sheaf cohomology and free resolutions over 
the exterior algebra}, Trans. Amer. Math. Soc. \textbf{355} (2003), no. 11, 4397--4426.

\bibitem[EPY]{epy}
D. Eisenbud, S. Popescu and S. Yuzvinsky, {Hyperplane arrangement cohomology and monomials in the exterior algebra}, Trans. Amer. Math. Soc. \textbf{355} (2003), no.11, 4365--4383.

\bibitem[EG]{eg}
E.G. Evans and P. Griffith, {The syzygy problem}, Ann. of Math. \textbf{114} (1981), 323-333.

\bibitem[GL1]{gl1}
M. Green and R. Lazarsfeld, {Deformation theory, generic vanishing
theorems, and some conjectures of Enriques, Catanese and Beauville},
Invent. Math. \textbf{90} (1987), 389--407.

\bibitem[GL2]{gl2}
M. Green and R. Lazarsfeld, {Higher obstructions to deforming
cohomology groups of line bundles}, J. Amer. Math. Soc. \textbf{1}
(1991), no.4, 87--103.

\bibitem[EGAII]{EGAII}
A. Grothendieck, {El\'ements de G\'eom\'etrie Alg\'ebrique II}, Publ. Res. Math. Inst. Hautes \'Etudes Sci., \textbf{8}, 1961.

\bibitem[EGAIII]{EGAIII}
A. Grothendieck, {El\'ements de G\'eom\'etrie Alg\'ebrique III}, Publ. Res. Math. Inst. Hautes \'Etudes Sci., \textbf{11}, 1961
and \textbf{17}, 1963.

\bibitem[Hac]{hacon}
Ch. Hacon, {A derived category approach to generic vanishing}, J.
Reine Angew. Math. \textbf{575} (2004), 173--187.

\bibitem[HP]{hp}
Ch. Hacon and R. Pardini, {On the birational geometry of varieties of maximal 
Albanese dimension},  J. Reine Angew. Math. \textbf{546} (2002), 177-199. 

\bibitem[Kl]{kleiman}
S. L. Kleiman, {The Picard scheme}, Fundamental algebraic geometry, 235--321, Math. Surveys Monogr. \textbf{123}, Amer. Math. Soc., Providence, RI, 2005.

\bibitem[Ko1]{kollar1}
J. Koll\'ar, {Higher direct images of dualizing sheaves I}, 
Ann. of Math. \textbf{123} (1986), 11--42.

\bibitem[Ko2]{kollar2}
J. Koll\'ar, {Higher direct images of dualizing sheaves II}, Ann. of
Math. \textbf{124} (1986), 171--202.

\bibitem[La]{positivity}
R.~Lazarsfeld, \emph{Positivity in algebraic geometry} I $\&$ II,
Ergebnisse der Mathematik und ihrer Grenzgebiete, vol. \textbf{48}
$\&$ \textbf{49}, Springer-Verlag, Berlin, 2004.

\bibitem[Lo]{lombardi}
L. Lombardi, in preparation.

\bibitem[MLP]{mlp}
M. Mendes Lopes and R. Pardini, {On surfaces with $p_g = 2q - 3$}, preprint arXiv:0811.0390.

\bibitem[MKR]{mkr}
N. Mohan Kumar and A. P. Rao, {Buchsbaum bundles on $\PP^n$}, J. Pure Appl. Algebra \textbf{152} (2000), 195--199.

\bibitem[Mu]{mukai}
 S. Mukai,  {Duality between $D(X)$ and $D(\widehat{X})$ with its
application to Picard sheaves}, Nagoya Math. J. \textbf{81} (1981),
153--175.


\bibitem[OSS]{oss}
Ch. Okonek, M. Schneider and H. Spindler, \emph{Vector bundles on complex projective spaces}, Birkh\"auser,  Boston, 1980.

\bibitem[Pa]{pareschi}
G. Pareschi, {Generic vanishing, Gaussian maps, and Fourier-Mukai transform}, preprint math.AG/0310026.

\bibitem[PP1]{PP1}
G. Pareschi and M. Popa, {$GV$-sheaves, Fourier-Mukai transform, and Generic Vanishing}, 
preprint math/0608127, to appear in Amer. J. Math.

\bibitem[PP2]{PP2}
G. Pareschi and M. Popa, {Strong Generic Vanishing and a higher dimensional Castelnuovo-de Franchis inequality},  
Duke Math. J. \textbf{150} no.2 (2009), 269--285.

\bibitem[RRV]{rrv}
J.-P. Ramis, G. Rouget and J.-L. Verdier, {Dualit\'e relative en g\'eom\'etrie analitique complexe}, 
Invent. Math. \textbf{13} (1971), 261--283.

\bibitem[Sa]{saito}
M. Saito, {Decomposition theorem for proper K\"ahler morphisms}, Tohoku Math. J. (2) \textbf{42},
(1990), no.2, 127--147.

\bibitem[GAGA]{GAGA}
J. P. Serre, {G\'eom\'etrie alg\'ebrique et g\'eom\'etrie analytique}, Ann. Inst. Fourier \textbf{6} (1955--1956), 1--42.

\bibitem[Ta]{takegoshi}
K. Takegoshi, {Higher direct images of canonical sheaves tensorized with semi-positive 
vector bundles by proper K\"ahler morphisms}, Math. Ann. \textbf{303} (1995), no.3, 389--416.

\end{thebibliography}
\end{document}